\documentclass[11pt,reqno]{amsart}
\usepackage{euscript,verbatim}

\newcommand{\eqdef}{\stackrel{\scriptscriptstyle\rm def}{=}}

\newtheorem{theorem}{Theorem}
\newtheorem{proposition}[theorem]{Proposition}

\newtheorem{corollary}[theorem]{Corollary}

\newtheorem{lemma}[theorem]{Lemma}

\newtheorem*{remarks}{Remarks}
\newcommand{\beha}{\begin{enumerate}}
\newcommand{\behe}{\end{enumerate}}
\renewcommand{\epsilon}{\varepsilon}

\newcommand{\ov}{\overline}

\newcommand{\bR}{{\mathbb R}}
\newcommand{\bC}{{\mathbb C}}

\newcommand{\bN}{{\mathbb N}}

\newcommand{\cH}{\EuScript{H}}
\newcommand{\cF}{{\mathcal F}}
\newcommand{\cA}{{\mathcal A}}
\newcommand{\cB}{{\mathcal B}}
\newcommand{\cC}{{\mathcal C}}
\newcommand{\cL}{{\mathcal L}}
\newcommand{\cX}{{\mathcal X}}
\newcommand{\cY}{{\mathcal Y}}
\newcommand{\cZ}{{\mathcal Z}}
\newcommand{\RE}[1]{(\ref{#1})}
\def\lumin{\lambda_u^{\text{\rm min}}}
\def\lumax{\lambda_u^{\text{\rm max}}}
\def\lsmin{\lambda_s^{\text{\rm min}}}
\def\lsmax{\lambda_s^{\text{\rm max}}}
\DeclareMathSymbol{\varnothing}{\mathord}{AMSb}{"3F}
\renewcommand{\emptyset}{\varnothing}
\thanks{This work was partially supported by a grant from the Simons Foundation (\#209846 to Christian Wolf).}

\begin{document}

\title{ DIFFEOMORPHISMS WITH STABLE MANIFOLDS AS BASIN BOUNDARIES}

\begin{abstract}
In this paper we study the dynamics of a family of diffeomorphisms in $\bR^2$ defined by
$
F(x,y)=(g(x)+h(y),h(x)),
$
where $ g(x) $ is a unimodal $C^2$-map which has the same dynamical properties
as the logistic map $P(x)=\mu x(1-x)$, and  $h(x) $ is a $C^2$ map which is a small perturbation of a linear map. For certain maps of this form  we show  that there are exactly two periodic points, namely an attracting fixed point and a saddle fixed point and the boundary of the basin of attraction is the stable manifold of the saddle.
 The basin boundary also has the same regularity as $F$, in contrast to the frequently observed fractal nature of basin boundaries. To establish these results we describe 
the orbits under forward and backward iteration of every point in the plane. 

\end{abstract}

\date{\today}

\author{Sandra Hayes} \address{The City College of New York}
\email{shayes@gc.cuny.edu}

\author{Christian Wolf}\address{The City College of New York}
\email{cwolf@ccny.cuny.edu}

\keywords{H\'enon-like maps,  diffeomorphisms in the plane, periodic orbits, stable manifolds}
\subjclass[2000]{Primary: 37C05, 37C25, 37E30}

\maketitle

\section{Notation and results}
If a diffeomorphism of $\bR^2$ has exactly two periodic points and one is attracting while the other one is a saddle point, it is of interest to find out when the boundary of the basin of attraction is the stable manifold of the saddle. Only in a few special cases of H\'enon maps (see \cite{HW}) this has  been proved, although in the standard literature on H\'enon maps (see for example the books \cite{ASY,R}) it is often said to be true based on computer experiments. 
In this paper we study $C^2$-diffeomorphisms  $F:\bR^2\to \bR^2$ defined by
\begin{equation}\label{eqgenfam}
F(x,y)=(g(x)+h(y),h(x)),
\end{equation}
where $ g(x) $ is a unimodal $C^2$-map which has the same dynamical properties
as the logistic map $P(x)=\mu x(1-x)$, and  $h(x) $ is a $C^2$-map with $h(0) = 0$ which is a small perturbation of a linear map. For certain maps of this type we show that $F$
has precisely two periodic orbits, one of which is an attracting fixed point and the other one is a saddle point. Moreover, the stable manifold of the saddle is the basin boundary of the attracting fixed point. To establish this result we  derive a  description of the dynamics under forward and under backward iteration for all points in the plane.

The strategy in this paper is to consider first a prototype family of diffeomorphisms of form \eqref{eqgenfam}. Namely, we consider a two-parameter family of H\'enon maps defined by
\begin{equation}\label{defhenon}
F_{\delta,\mu}(x,y)=(\mu x (1-x)+\delta y, \delta x),
\end{equation}
where $\mu$ and $\delta$ are real parameters with $\mu>0$ and $\delta\not=0$.
We then develop for certain parameters $\mu$ and $\delta$ a complete description of all possible orbits under forward as well as backward iteration. Finally, we consider more general maps of the form \eqref{eqgenfam} and explain how to obtain results similar to those for H\'enon maps.
 
For certain parameters  $(\delta,\mu)$ for which $F_{\delta,\mu}$ has exactly two periodic points, namely an attracting fixed point and a saddle fixed point, we show that the boundary of the basin of attraction is the stable manifold of the saddle. Our main result for H\'enon maps is the following.
\begin{theorem}
For certain parameters $(\delta,\mu)$ the following holds.
\begin{enumerate}\label{thmainthm}
\item[(i)] The fixed points $\alpha$ and the origin 
$(0,0)$ are the only periodic points of $F = F_{\delta,\mu}$. 
\item[(ii)]
 The boundary  $\partial W^s(\alpha)$ of the basin of  $\alpha$ is the stable manifold $W^s(0,0)$ of the origin.
  \item[(iii)]
 The filled Julia set $K$ of points with bounded forward and backward orbits for $F$ is
 $\{\alpha,(0,0)\}\cup (W^s(\alpha)\cap W^u(0,0))$.
\end{enumerate}

 \end{theorem}
 The origin is also an unstable one-sided flip saddle point meaning one eigenvalue is negative and only one connected component of the punctured unstable manifold $W^{u}(0,0)\setminus (0,0)$ of $(0,0)$ meets $K$. \\

An immediate consequence of this theorem is the following.
\begin{corollary}\label{cormosm}
For each $(\delta,\mu)$ as in Theorem 1 the map $F_{\mu,\delta}$ is  a Morse-Smale diffeomorphism.
\end{corollary}

Note that $F_{\delta,\mu}$ is conjugate to $F_{-\delta,\mu}$ via $G(x,y)=(x,-y)$. Therefore, for the area decreasing case $0 \not = |\delta| < 1$ we only need to consider  $0 < \delta <1$. The area increasing case  will have a similar description due to the conjugacy of $F^{-1}$.  Thus, for the non-area preserving case we only have to consider parameters for which $F_{\mu,\delta}$ has two distinct fixed points, namely we consider parameters in
\begin{equation}\label{H11}
H_{1,1} = \{ (\delta, \mu):  0 < \delta < 1 \quad and\quad  1- \delta^2  < \mu < 3( 1 - \delta^2)\}
\end{equation}

Next we consider the general  case of diffeomorphisms of the form \eqref{eqgenfam}. We assume that $g:\bR\to \bR$ is a $C^2$-unimodal map with $g(0)=0, g(1)=0, g([0,1])\subset [0,1)$, $g'(0)>1$ and $g'(1)<-1$. We also assume that there exists $\gamma<0$ such that $g''(x)<\gamma$ for all $x\in \bR\setminus (0,1)$. Finally, we assume that there exists an attracting fixed point $x=x_g\in (0,1)$ of $g$ with $W^s(x)=(0,1)$.  For $\delta\in \bR$ we denote by $L_\delta$ the linear map $x\mapsto \delta x$.
Moreover, if $f:\bR\to\bR$ is a $C^2$-map, we write $||f||_2=\sup_{x\in\bR}\{|f(x)|,|f'(x)|,|f''(x)|\}$. Note that $||.||_2$ may be infinite and is, in particular  not a norm.
For maps $g$ with the properties above we have the following result.
\begin{theorem}\label{thmgeneral}
There exists $\delta(g)>0$ such that for all $0<\delta<\delta(g)$ there is $\epsilon>0$ such that for any $C^2$-map $h:\bR\to \bR$ with $h(0)=0$ and $||h-L_\delta||_2<\epsilon$ the following holds.
\begin{enumerate}
\item[(i)] 
The map $F(x,y)=(g(x)+h(y),h(x))$ is a $C^2$-diffeomorphism of $\bR^2$.
\item[(ii)]
The map $F$ has precisely two periodic points, both of which are fixed points.  The origin 
$(0,0)$ is a saddle point and the other fixed point $\alpha$ is attracting.
\item[(iii)]
 The boundary  $\partial W^s(\alpha)$ of the basin of  $\alpha$ is the stable manifold $W^s(0,0)$ of the origin.
  \item[(iv)]
 The  set $K$ of points with bounded forward and backward orbits for $F$ is
 $\{\alpha,(0,0)\}\cup (W^s(\alpha)\cap W^u(0,0))$.
\end{enumerate}

 \end{theorem}

\section{The dynamics of $F_{\mu,\delta}$ for small $\delta$}
The purpose of this section is to prove Theorem \ref{thmainthm}. We construct a partitioning of $\bR^2$ and use it to investigate the forward and backward orbits of every point. A key ingredient of the proof is that the dynamics of the maps $F_{\delta, \mu}$ is controlled by the set $K$ of points with bounded (forward and backward) orbits. Unless stated otherwise we use the maximum norm in  $\bR^2$. Points with backward (resp. forward ) orbits escaping to infinity under this norm will simply be denoted by $W^u(\infty)$ (resp. $W^s(\infty)$).
Throughout this section we use as a standing assumption that $(\delta,\mu)\in H_{1,1}$ defined in \eqref{H11}. We begin by listing some elementary properties of  $F_{\delta, \mu}$ that follow from straight forward  calculations.
Let  $ F=F_{\delta,\mu}$ be as in \eqref{defhenon}. Then
\begin{enumerate}
\item[(i)]  The eigenvalues of the Jacobian matrix 
\begin{equation}
DF(x, y )= 
\left( \begin{array}{ccc}
\mu - 2\mu x & \delta  \\
\delta & 0 \\
 \end{array} \right)
\end{equation}
are given by 
\begin{equation}  
\lambda_1 = \lambda_1(x,\delta, \mu)= \frac{(1 - 2x)\mu - \sqrt{\mu^2(2x-1)^2 + 4 \delta^2}}{2},
\end{equation} and
\begin{equation} 
\lambda_2 = \lambda_2(x, \delta, \mu) 
 \frac{(1 - 2x)\mu + \sqrt{\mu^2(2x-1)^2 + 4 \delta^2}}{2}.
\end{equation}
\item[(ii)]
We have that $\lambda_1 < 0 $ and $ \lambda_2 > 0 $ for every $x $ and $ \mu$ when $\delta \not= 0$.
\item[(iii)] 
For the origin $(0,0)$, $\lambda_2 > 1$ if and only if $\delta^2 > 1-\mu$ and $\lambda_1 > -1$ if and only if $\delta^2< 1+ \mu$.
\item[(iv)]  Let $\mu > 1$ and $\delta \not= 0$. Then the origin is a flip saddle fixed point for $F$ if and only if 
$\delta^2 < 1 + \mu$.
\item[(v)] If $1 < \mu < 3$, then the fixed point $\alpha = (x_{\alpha}, y_{\alpha}) = (x_\mu + \frac{\delta^2 }{\mu}, \delta x_\alpha) = (1- \frac{1}{\mu} + \frac{\delta^2}{\mu}, \delta x_\alpha)$ is attracting if and only if  
$\mu < 3(1-\delta^2).$
\item[(vi)] The inverse  of $F$ is given by the formula $F^{-1}(x,y)=$
\begin{equation}
(\delta^{-1}y, \delta^{-1}[x-P_\mu(\delta^{-1}y)])=
(\delta^{-1}y,\delta^{-1}[x-\mu \delta^{-1}y(1-\delta^{-1}y)]),
\end{equation}
where $P_\mu(x)=\mu x(1-x)$.
\end{enumerate}
Our goal is to partition the region above the line $y = 2\delta$  into 3 parts. Similarly for the region below $y = - 2\delta$.  We define
\[
W_\delta=\{(x,y):  |y|\geq 2\delta, |y|\geq \delta|x|\}.
\]
\begin{lemma}\label{lem4}
Let $1<\mu<3$ and let $0<\delta<\max\{\mu-1,1\}.$  
Then $W_{\delta}\subset W^u(\infty)$.
 \end{lemma}
\begin{proof}
Let $(x,y)\in W_\delta$ and  $F^{-n}(x, y) = (x_{-n}, y_{-n}).$ It will be shown that $|y_{-n}|\to \infty$ as $n\to\infty$.
\begin{equation}\label{kmin}
\begin{split}
|y_{-1}| = & \left| \delta^{-1}[x-\mu \delta^{-1}y(1-\delta^{-1}y)]\right|\\
\geq & \delta^{-1} \left[ \mu |\delta^{-2} y^2| - |x|-\mu |\delta^{-1} y |\right]\\
\geq& \delta^{-1}\  |y| \left[ 2\mu \delta^{-1}  - \delta^{-1}-\mu \delta^{-1}  \right]\\
\geq &  \delta^{-1}\  |y| \left[\delta^{-1}\mu-\delta^{-1}\right]\\
\geq & \delta^{-1}\  |y|\ [\delta^{-1}(\mu-1)] > |y|,
\end{split}
\end{equation}
since   $\delta^{-1}(\mu-1)>1$ when $0<\delta<\mu-1$. Recall that $x_{-1}=\delta^{-1}y$. Therefore, \eqref{kmin} implies that $(x_{-1},y_{-1})\in W_\delta$. It now follows by induction (using \eqref{kmin})  that the sequence  $(|y_{-n}|)_n$ is strictly increasing and $|y_{-n}|\to \infty$ as $n\to\infty$.
\end{proof}
The remaining two parts of the region above  $y = 2\delta$ will be treated separately.  That part below $y = -\delta x$ for  $x\leq -2$ is in $W^s(\infty)$  as follows from the next Lemma when $\delta$ is small enough.
\begin{lemma}\label{betalem}
Let $1<\mu<3$. Define $\beta=\beta(\delta)=\delta (\mu-1)^{-1}$. Then for all $(x,y)\in \bR^2$ with $x\leq \min\{-\beta y, 0\}$ and $(x,y)\not=(0,0)$ we have $(x, y) \in W^s(\infty)$. \end{lemma}
\begin{proof}
Let $(x,y)\in \bR^2$ with $x\leq \min\{-\beta y, 0\}$. First,  we consider  $y\geq 0$ in which case $x\leq -\beta y$. We will show that $x_{n}\to -\infty$ as $n\to \infty$ for  $(x_n,y_n) = F^n(x,y).$ Since
\begin{equation}
x_{1}=\mu x- \mu x^2 + \delta y \leq x-\mu x^2+[\mu-1 -\delta \beta^{-1}] x\leq       x-\mu x^2
\end{equation}
that follows by induction. The case $y<0$ follows similarly with the simplification that $y_{n}<0$ for all $n\in\bN$. 
\end{proof}
It follows from Lemma \ref{betalem} that every point $(x, y)$ with $x < -2$ and $2\delta < y < -\delta x$ is in $ W^s(\infty)$ when $\delta < \sqrt{\mu-1}.$
The region above $y = 2\delta$ and below $y = \delta x$ for $x > 2$ is in $W^s(\infty)$ as follows from the next Lemma.
\begin{lemma}\label{lem5}
Let $1<\mu<3$ and let $0<\delta<\delta^*\eqdef\sqrt{\mu-1}$. Let $ x\geq 2, 0\leq y \leq \delta x$. Then $(x, y) \in W^s(\infty)$.
\end{lemma}
\begin{proof} 
Let $\beta$ be as in Lemma \ref{betalem}. Then
\begin{equation}\label{asd1}
x_1=\mu x(1-x)+\delta y \leq -\mu x +\delta^2 x \leq -\mu x + (\mu-1)x= -x
\end{equation} 
On the other hand,
\begin{equation}\label{asd2}
-\beta y_1=-\frac{\delta}{\mu-1} \delta x \geq -x \geq x_1.
\end{equation}
Thus,  the point $(x_1,y_1)$ satisfies the conditions of Lemma \ref{betalem} and  $(x_1, y_1) $ is in $ W^s(\infty)$ .
\end{proof}
From Lemmas 4, 5 and 6 we conclude that above the line $y = 2\delta$ either the backward or the forward orbits escape to $\infty$. Therefore, $K$ must lie below $y = 2\delta$.

Now we will partition the strip $-\infty < x < \infty, 0 < y < 2\delta$ and investigate the orbits in each partition. Note that the region below $y = 2\delta$ and above the $x$-axis for $x\geq2$ is in $W^s(\infty)$ by Lemma \ref{lem5}. Next, we consider
\begin{equation}\label{Adelta}
A_\delta=\{(x,y): 0\leq x\leq 1, 0\leq y\leq 2\delta\}\setminus \{(0,0)\}
\end{equation}
for $\delta > 0$ and show that it is in the basin $W^s (\alpha)$ of attraction. We will show the existence of a polydisk with center $\alpha$ that is contained in $W^s(\alpha)$ and whose size is independent of $\delta$ and $\mu$.
\begin{proposition}\label{lem1}
Let $1<\mu<3$. Then there exists $r>0$ and $\delta^*>0$ such that for all $0<\delta<\delta^*$
we have $\overline{P}(\alpha,r)=\{(x,y)\in \bR^2: |x-x_\alpha|, |y-y_\alpha|\leq r\} \subset W^s(\alpha)$.
\end{proposition}
\begin{proof}
Let $(x,y)=(x_0,y_0)=(x_\alpha+s_0,y_\alpha+t_0)$ and $(x_n,y_n) = F^n(x_0,y_0)= (x_\alpha+s_n,y_\alpha+t_n)$. Thus $(x,y)\in W^s(\alpha)$ if and only if $||(s_n,t_n)||_m \to 0$ as $n\to \infty$ where the maximum norm is used. Let $n\in \bN$. Then 
\[
\begin{split}
(x_{n+1}, y_{n+1}) = & F(x_\alpha+ s_n, y_\alpha+t_n)\\
  =& (\mu (x_\alpha + s_n)(1-(x_\alpha +s_n)) + \delta (y_\alpha + t_n), \delta (x_\alpha + s_n))\\
=   &  ( \mu x_\alpha (1- x_\alpha) +\delta y_\alpha + s_n (\mu -2\mu x_\alpha - \mu s_s) + \delta t_n, \delta x_\alpha + \delta s_n)\\
= & (x_\alpha + s_n (2-\mu - \frac{2\delta^2}{\mu} -\mu s_n)+ \delta t_n, y_\alpha +\delta s_n).
\end{split}
\]
Hence,
\begin{equation}\label{eqgeo1}
(x_{n+1},y_{n+1})-(x_\alpha, y_\alpha)=(s_n (2-\mu - \frac{2\delta^2}{\mu} -\mu s_n)+ \delta t_n,\delta s_n).
\end{equation}
Since  $2-\mu \in (-1,1)$ there exists $0<\gamma<1$ and $r>0$ such that $|2-\mu - \mu s|< \gamma$ for all $|s|\leq r$.
Moreover, an elementary continuity argument shows that there exists $\delta^*>0$ such that for all $0<\delta< \delta^*$ and all $|s|\leq r$ we have
\begin{equation}\label{eqgeo2}
 \left|2-\mu - \frac{2\delta^2}{\mu} -\mu s\right|< \gamma.
\end{equation}
Without loss of generality assume that $\delta^*< \frac{1-\gamma}{2}$. 
We conclude that if $0<\delta<\delta^*$ and $(x_0,y_0) \in \overline{P}(\alpha,r)$
then $||(s_n,t_n)||_m$ converges geometrically to $0$ and thus $(x_0,y_0)  \in W^s(\alpha)$.
\end{proof}
As a consequence of Proposition \ref{lem1} we obtain the following.
\begin{corollary}\label{lemkeg}
Let $1<\mu<3$ and let $h>0$. Then there exist $r>0$ and $\delta^*>0$ such that if $ 0<\delta<\delta^*$, $x_\mu-r\leq x \leq x_\mu+r$ and $0\leq |y|\leq h$ then $(x,y)\in W^s(\alpha)$.
\end{corollary}
\begin{proof}
Note that the radius $r$ in Proposition \ref{lem1} only depends on $\delta^*$ and not on $\delta$. Moreover, $x_\alpha\to x_\mu$ and $y_\alpha\to 0$ as $\delta\to 0$.
Therefore the claim follows immediately from  Proposition 7 and the fact that we can make $\delta y$ as small as necessary by making $\delta^*$ small.
\end{proof}
Note that in Corollary \ref{lemkeg} the value of $\delta^*$ can be chosen to depend continuously on $\mu$.

\begin{proposition}\label{propA}
Let $1<\mu<3$. 
Then there exists $\delta^*>0$ such that for all $ 0<\delta<\delta^*$
we have $A_\delta\subset W^s(\alpha)$.
\end{proposition}
\begin{proof}
Let $\delta^*, h$ and $r$ be as in Corollary \ref{lemkeg}. Moreover, we may assume  that $\delta^*<\frac{1}{4}$. Let $0<\delta<\delta^*$.
Obviously, $P_\mu([0,1])\subset [0,\frac{3}{4})$ which implies that $F(A_\delta)\subset A_\delta$.
We have $P_\mu'(0)=\mu>1$. Thus, by continuity there exists
 $\rho_l>0$ such that $P_\mu'(x)\geq \frac{\mu+1}{2}>1$ for all $0\leq x\leq \rho_l$. Since $x_\mu$ is an attracting fixed point of $P_\mu$ we must have $\rho_l<x_\mu$.
Let $(x,y)\in A_\delta$ with $0\leq x\leq \rho_l$. If $x=0$ then $x_1>0$ and therefore we may assume  that $x>0$. It follows from the Mean Value Theorem
that
\begin{equation}
x_1=P_\mu(x)+\delta y\geq P_\mu(x)-0\geq \frac{\mu+1}{2} x.
\end{equation}
It now follows by induction that there exists $n\in \bN$ such that $x_n\geq \rho_l$. \\
\noindent
On the other hand, if $(x,y)\in A_\delta$ then  
\begin{equation}
x_1=P_\mu(x)+\delta y< \frac{3}{4}+2\delta^2<\frac{7}{8}.
\end{equation}
We define $\rho_u=\frac{1}{8}$. The above shows that in order to  prove the claim it is  sufficient  to consider $(x,y)\in A_\delta$ with $\rho_l\leq x\leq \rho_u$.
First, we  consider the dynamics of $P_\mu$ on  $[\rho_l,\rho_u]$.
Since each $x\in [\rho_l,\rho_r]$ is contained in $W^s(x_\mu)$ (with respect to $P_\mu$) and since  $[\rho_l,\rho_r]$ is compact there exists $m\in \bN$ such that $P_\mu^m(x)\in (x_\mu-r,x_\mu+r)$ for all 
$x\in [\rho_l,\rho_r]$. Note that
\begin{equation}
F^m(x,y)=(P_\mu^m(x)+P(x,y),Q(x,y)),
\end{equation}
for some polynomials $P$ and $Q$ in two variables. Moreover, each of the coefficients of $P$ and $Q$ contains positive powers of $\delta$.
Thus, by making $\delta^*$ smaller if necessary, we can assure that if $(x,y)\in A_\delta$ with $\rho_l\leq x\leq \rho_u$ then $x_m\in (x_\mu-r,x_\mu+r)$ and the claim follows.
\end{proof}
\noindent
Next we will investigate the negative part of the strip
\begin{equation}
 S_{\delta} = \{(x, y): -\infty < x < \infty, 0 < y < 2\delta\},
 \end{equation}
 i.e. $x \leq 0, 0< y < 2\delta$, which contains a portion of $W^s(0, 0)$. In particular, we show that the local stable manifold of the orgin is contained in the boundary of the basin of attraction of $\alpha$.
\begin{theorem}\label{thwsloc}
 Let $1<\mu<3$ and let $\beta=\delta(\mu-1)^{-1}.$ Then  there exists $\delta^*>0$ such that for all $ 0<\delta<\delta^*$  
the following holds:
\begin{enumerate}
\item[(i)]
For all $0\leq \ov{y} \leq 2\delta$ there exists a unique
$ -\beta \ov{y} \leq \ov{x}\leq 0$ such that $(\ov{x},\ov{y})\in W^s(0,0)$.
\item[(ii)]
If $0< \ov{y}\leq 2\delta$ then $-\beta \ov{y}<\ov{x}<0$;
\item[(iii)]
Let $0< \ov{y}\leq 2\delta$. If $\ov{x}<x\leq 0$ then $(x,\ov{y})\in W^s(\alpha)$, and if $x<\ov{x}$ then $(x,y)\in W^s(\infty)$.
\end{enumerate}
\end{theorem}
\begin{proof}
Let $\delta^*$ be as in Proposition \ref{propA}. 
We first prove the existence in (i). The uniqueness in (i) will follow from (iii). If $\ov{y}=0$ then $\ov{x}=0$. Assume now $0< \ov{y}\leq 2\delta$.
We define
\begin{equation}
\ov{x}=\inf\{x\leq 0: (x,0]\subset W^s(\alpha)\}.
\end{equation}
By Proposition \ref{propA}, $(0,\ov{y})\in W^s(\alpha)$. Since $W^s(\alpha)$ is open we may conclude that $\ov{x}<0$. It follows from the definition of $\ov{x}$
that $ (\ov{x},\ov{y})\in \partial W^s(\alpha)$. 
Moreover, since $W^s(\infty)$ is  open we obtain $  (\ov{x},\ov{y})\not\in 
W^s(\infty)$.
Combining this with Lemma \ref{betalem} yields $\ov{x}>-\beta \ov{y}$. Thus, (ii) holds.
For $n\in \bN$ we write $(x_{n},y_n)=F^n( \ov{x},\ov{y})$. In particular,
\begin{equation}\label{eq2n}
(x_{2n},y_{2n})=(P_\mu(P_\mu(x_{2n-2})+\delta y_{2n-2})+\delta^2 x_{2n-2}, \delta(P_\mu(x_{2n-2})+\delta y_{2n-2})).
\end{equation}
We define $C=C(\delta)=\sup \{|P'_\mu(x)|: (x,y)\in C_\delta\}$. Clearly, $C(\delta)\to \mu$ as $\delta\to 0$. Since $\beta\to 0$ as $\delta\to 0$,
 we can assure (by making $\delta^*$ smaller if necessary) that 
\begin{equation}
\delta(C\beta+\delta)<\frac{1}{2}.
\end{equation}
We claim that $(x_{2n},y_{2n})\in C_\delta$ for all $n\in \bN_0$. Since $(x_0,y_0)=(\ov{x},\ov{y})$ the claim holds for $n=0$. Assume now
that the claim holds for $n-1$. Then $y_{2n}\geq 0$.  Otherwise $(x_{2n},y_{2n})$ would be contained in the third quadrant and thus by Lemma
\ref{betalem}   in $W^s(\infty)$. We obtain 
\begin{equation}\label{eqconv}
\begin{split}
|y_{2n}|&=\delta \ |P_\mu(x_{2n-2})+\delta y_{2n-2}|\leq \delta\ (|Cx_{2n-2}|+|\delta y_{2n-2}|)\\
& \leq \delta (C\beta+\delta)\ |y_{2n-2}|<\frac{1}{2}\ |y_{2n-2}|.
\end{split}
\end{equation}
Finally, if $x_{2n}< -\beta y_{2n}$ then by Lemma \ref{betalem} we would have $(x_{2n},y_{2n})\in W^s(\infty)$ which is a contradiction to 
$(x_{2n},y_{2n})\in \partial W^s(\alpha)$ and the claim is proved.
Since  $-\beta y_{2n}\leq x_{2n}\leq 0$, equation \eqref{eqconv} proves that $(\ov{x},\ov{y})\in W^s(0,0)$.

It remains to prove (iii). Let $0<y\leq 2\delta$.
That  $(\ov{x},0]\times \{\ov{y}\}\subset W^s(\alpha)$ is a direct consequence of the definition of $\ov{x}$. 
We now consider the case $x<\ov{x}$. If $x<-\beta y$ then $(x,y)\in W^s(\infty)$ by Lemma \ref{betalem}.
Assume now $(x,y)\in C_\delta$. It follows from \eqref{eq2n} (with $(\ov{x},\ov{y})$ replaced by $(x,y)$) and Lemma \ref{betalem} that
if $y_{2n}<0$ 
for some $n\in \bN$ then $(x,y)\in W^s(\infty)$. It remains to consider the case $y_{2n}\geq 0$ for all $n\in \bN$.
Similar to the case of $(\ov{x},\ov{y})$ it follows from \eqref{eqconv}  that $y_{2n}\to 0$ as $n\to\infty$. 
Note that $\inf \{|P'_\mu(x)|: (x,y)\in C_\delta\}=\mu>1$. It now follows by induction and by using \eqref{eq2n} 
that $x_{2n}< -\beta y_{2n}$ for some $n\in \bN$. Therefore, Lemma \ref{betalem} implies $(x,y)\in W^s(\infty)$ and the proof is complete.
\end{proof}
As a result of the Theorem \ref{thwsloc}, all points $(x, y) $ with $x <0$ and $0< y < 2\delta$ have forward orbits which escape to $\infty$ or converge to the attracting fixed point $\alpha$, depending on whether they lie to the left of $W^s(0,0)$ or to the right. The only partition of the strip $-\infty < x < \infty, 0 < y < 2\delta$ not dealt with yet
is the set of points $(x, y)$ with $ 1 < x < 2, 0 < y < 2\delta$. It will be shown that those points have forward orbits converging to  $\alpha$ or escaping to $\infty$ according to whether they lie to the left of $W^s(0,0)$  or to the right.\\
Before we attack that proof, an interesting property of the map $\ov{y}\mapsto \ov{x}$ given in Theorem \ref{thwsloc}  can be derived.

Note that for small $y$ values part (i) of Theorem \ref{thwsloc} also follows from the Stable Manifold Theorem applied to the saddle point $(0,0)$. For our purposes
however, it is crucial to have a uniform estimate from below for the size of the local stable manifolds. On the other hand, the Stable Manifold Theorem
implies the following:
\begin{corollary}
Under the assumptions of Theorem \ref{thwsloc} the map $\ov{y}\mapsto \ov{x}$ is real analytic.
\end{corollary}
\begin{proof}
By the Stable Manifold Theorem there exists $\eta>0$ such that 
\[W^s_{\rm loc}(0,0)\cap \{(x,y): y\geq 0\}=\{(\ov{x},\ov{y}): 0\leq \ov{y}< \eta\}.
\]
The statement now follows from the fact that
\[
F^{2n}(\{(\ov{x},\ov{y}): 0\leq \ov{y}\leq 2\delta\})\subset  
W^s_{\rm loc}(0,0)\cap \{(x,y): y\geq 0\}
\]
for some $n\in \bN$.
\end{proof}
Finally, we can treat the last part of the strip $S_{\delta}$, namely points $(x, y)$ with  $1 \leq x \leq 2, 0 < y < 2\delta $:  
\begin{theorem}\label{thws2}
 Let $1<\mu<3$. Then  there exists $\delta^*>0$ such that for all $ 0<\delta<\delta^*$  
the following holds:
\begin{enumerate}
\item[(i)]
For all $0\leq \ov{y} \leq 2\delta$ there exists a unique
$1<\ov{x}<2$ such that $(\ov{x},\ov{y})\in W^s(0,0)$. Morover, $\ov{y}\mapsto \ov{x}$ is real-analytic.
\item[(ii)]
Let $0\leq \ov{y}\leq 2\delta$. If $1\leq x<\ov{x}$ then $(x,\ov{y})\in W^s(\alpha)$, and if $\ov{x}<x\leq 2$ then $(x,\ov{y})\in W^s(\infty)$.
\end{enumerate}
\end{theorem}
\begin{proof}
The rectangle $B_{\delta} = \{ (x, y): 1 \leq x \leq 2, 0 \leq y \leq 2\delta\}$  for $ \delta >0$ has an image $F(B_{\delta})$ which is a topological rectangle. Let $ 0\leq \ov{y} \leq 2 \delta$ be fixed. The curve $F(x, \ov{y})$, $1\leq x \leq 2$, which is the image of a horizontal line, is a parabola in that topological rectangle opening to the left. It is parallel to the parabolas which are the images of $F(x, 0)$ and $F(x, 2\delta)$, $1\leq x \leq 2$ . Let $Q =  W^s(0, 0) \cap F(\ov{x}, \ov{y})$ be the intersection of the parabola  $F(x, \ov{y})$, $1\leq x \leq 2$, and the local stable manifold of the origin. Then $Q' = F^{-1}(Q) = (\ov{x}, \ov{y})$ is in $ W^s(0, 0)$. If $1\leq x < \ov{x}$, then the $x-$coordinate of $Q_1 = F(x, \ov{y})$ is larger than that of $Q$, whereas its $y$-coordinate is smaller than that of $Q$, implying it is to the right of $W^s(0, 0)$ and thus in the basin of attraction $W^s(\alpha)$ by Theorem \ref{thwsloc}. Similarly, if $\ov{x} < x \leq 2$, then $(x, \ov{y})$ is in $W^s (\infty)$, since the $y$-coordinate of $Q_1$ is larger than that of $Q$ whereas the $x$-coordinate is smaller than that of $Q$.
\end{proof}
Using the previous results, the filled Julia set $K$ for the map $F_{\delta, \mu}$ must be below the line $y = 2 \delta$.  The behavior of all forward orbits of points in the strip $S_{\delta} $ can be described for small $\delta$ as follows.
\begin{proposition}\label{propallorbitsS}
Let $1 < \mu < 3$. There exists a $\delta^* > 0$ such that for all $0 < \delta < \delta^*$ the forward orbit under $F_{\delta, \mu}$ of every point $(x, y)$ with  $-\infty < x < \infty, 0 < y < 2\delta$
converges either to $(0,0)$, to $\alpha$ or to $\infty$.
\end{proposition}
\begin{proof}
Theorem \ref{thwsloc} describes the behavior of forward orbits for points  with $x \leq 0$.
For $0 < x \leq 2$, the statement follows from Proposition \ref{propA} and Theorem \ref{thws2}. Finally, Lemma \ref{lem5} establishes the claim for $x > 2$.
\end{proof}
To summarize, for all points in the upper half plane we know the fate of either the forward or the backward orbits. Now we will concentrate on the points below the $x$-axis. As a result of Lemma \ref{betalem}, the third quadrant is in $W^s(\infty)$. We now investigate the other possibilities.

\begin{proposition}\label{prop16} Let $1<\mu< 3$. Then
\begin{enumerate}
\item[(i)] If $0<\delta<\delta^*\eqdef\sqrt{3\mu(\mu-1)}$ and $(x,y)\in \bR^2$ with $x\geq 2$ and $y\leq 0$ then $(x,y)\in W^s(\infty)$.
\item[(ii)]  If $0<\delta\leq 1$ and $(x,y)\in \bR^2$ with $0\leq x\leq 2$ and $y\leq \-\delta_0 =  -\delta^{-1}\left(\frac{2\delta^2}{\mu-1}+1\right)$ then $(x,y)\in W^s(\infty)$.
\end{enumerate}
\end{proposition}
\begin{proof}

(i) Suppose $x\geq 2$ and $y\leq 0$. Then $P_\mu(x)\leq -\mu x$. Hence,
\[
x_1=P_\mu(x)+\delta y \leq -\mu x < - \frac{\delta^2}{\mu-1} x= -\beta
\delta x=-\beta y_1.
\]
By Lemma \ref{betalem}, $(x_1,y_1)=F(x,y)\in W^s(\infty)$ which
proves (i).\\
\noindent
(ii) Let $(x,y)\in \bR^2$ with $0\leq x\leq 2$ and $y\leq
-\delta^{-1}\left(\frac{2\delta^2}{\mu-1}+1\right)$. Then
\[
x_1=P_\mu(x)+\delta y\leq 1+ \delta y \leq
-\frac{2\delta^2}{\mu-1}\leq -\beta \delta x= -\beta y_1.
\]
Therefore, the same argument as in the proof of (i) shows $(x,y)\in
W^s(\infty)$.
\end{proof}
In order to determine the behavior of the forward orbits for the remaining part of the fourth quadrant, 
we now investigate points $(x, y)$ on the stable manifold of the origin satisfying $0\leq x \leq 2$.

In the following we introduce some notation that will be used later.
Let $Q =  W^s(0, 0) \cap F(\ov{x},0)$ be the intersection of the parabola  $F(x, 0), 1<x<2, $ and the local stable manifold of the origin. Let  $W'$ be the part of  $W^s(0, 0)$ connecting $Q$ to the origin and let $Q' = F^{-1}(Q) = (\ov{x}, 0)$.  Then $C = F^{-1}(W')$ is a curve connecting  $Q' $ to the origin and is contained
in $W^s(0, 0)$. 
We claim that except for the two endpoints,  the curve $C$ is below the x-axis. To show the claim, we consider first  $(x, y) \in W'$ with $y \leq \delta$. Because $x < 0$ and  $F^{-1}(x, y) = \frac{1}{\delta}(y, x + \frac{\mu y(y - \delta)}{\delta^2})$, it follows that  $F^{-1}(x, y)$ is below the x-axis.
In particular,  $P' = F^{-1} (P) = (1, \frac{x}{\delta}) \in C$ for $P = (x, \delta) \in W'$  is below the x-axis as well as on $C$. 
If there would exist a point $(x, y)$ on $C$ with $y > 0 $ and $1< x < \ov{x}$, then
by the Intermediate Value Theorem there would have to be a point on the interval $(1,x )$ which is contained in $W^s(0, 0)$. But this is a contradiction, since
the entire interval $(0, \ov{x})$  is in $W^s(\alpha)$ by Proposition \ref{propA} and Theorem \ref{thwsloc}. 
\begin{lemma}\label{lem17} The region below $C$ in the fourth quadrant  is in $W^s (\infty)$. The region above $C$ and below the interval $(0, \ov{x})$ on the $x$-axis is contained in $W^s(\alpha)$.
\end{lemma}
\begin{proof}
Let $0<x \leq \ov{x} $ be fixed with $Q' = (\ov{x}, 0)$ on $W^s(0,0)$.   The vertical half-line $(x, y), y \leq 0 $ is mapped by $F$ to the horizontal half-line  $(z, \delta x), z\leq P_\mu(x)$. Let  $(x, y_0) $ be the intersection of $C$ and the vertical  half-line. Then  $( x, y_0)$  and   $F(x, y_0) $ are  on  $W^s(0, 0)$.
The horizontal half-line starts in the basin of attraction $W^s(\alpha)$. All points $(x, y)$ with  $ y_0 < y < 0$ are also in that basin by Theorem 12, since  the first coordinate of $F(x, y) = (P_\mu(x) + \delta  y, \delta x)$ is larger than that of $F(x, y_0)$. Similarly, if $y<y_0$, then all points $(x, y)$  are in $W^s (\infty)$ by Theorem \ref{thwsloc}.
\end{proof}
We are now able to provide a description of all forward orbits of   points below the $x$-axis for $1 < \mu < 3$ and  small $\delta$.

\begin{proposition} \label{prop17} There exists  $\delta^* > 0$ such that for all $0 < \delta < \delta^*$ the forward orbit under $F_{\delta, \mu}$ of every point $(x, y)\in \bR^2 $ with $-\infty<x<\infty, y<0$  converges either to $(0,0),$ to $\alpha$ or to $\infty$.
\end{proposition}
\begin{proof}
All points in the third quadrant  escape to infinity under forward iteration by Lemma \ref{betalem}. If $0<x<\ov{x}$, points $(x, y)$ with $y<0$ which are not on $C$ are either in the basin of attraction of $\alpha$ or escape to infinity according to whether they are above $C$ or below $C$ by Lemma \ref{lem17}. Obviously, points on $C$ converge to the origin. Theorems \ref{thwsloc} and  \ref{thws2} show that points $(x, y)$ for $ \ov{x} <  x< 2 $ and $y < 0$ also escape to infinity.
\end{proof}
It is now obvious that the forward orbit of every point in the real plane under $F_{\delta, \mu}$ converges either to 
$(0,0)$, to $\alpha $ or to $\infty$ for  $1<\mu <3$ and $\delta=\delta(\mu)$ small.

\begin{lemma} \label{lem19} Let $1 < \mu < 3$. Then there exists  $\delta^* > 0$ such that for all $0 < \delta < \delta^*$.
the map $F_{\delta, \mu}$ satisfies properties (i) and (ii) of Theorem 1.
\end{lemma}

\begin{proof}
(i) Let $1 < \mu < 3$ and  let $r$ be as in Proposition \ref{lem1} with $r < 1$. If $0 < \delta < \delta^*$ for  $\delta^*$  as in Propositions \ref{propallorbitsS} and \ref{prop17} and  $(\delta,\mu')\in (\bR\setminus \{0\})\times \bR$ with $||(\delta,\mu')-(0,\mu)||<r$. 
Then every forward orbit of $F_{\delta, \mu'}$ converges to either to $0$, to  $\alpha$  or to $\infty$. Consequently, there are no periodic points other than the two fixed points $(0,0)$ and $\alpha$.\\
(ii) It suffices to show that in a neighborhood $U$ of the origin, the boundary  $\partial W^s(\alpha)$ of the basin of  $\alpha$ is the stable manifold $W^s(0,0)$ of the origin. Without restriction every point in $U$ converges under forward iteration either to $0$, to $\alpha$  or to $\infty$. Let  $(x, y)$  be in  $U$ and on  $W^s(0,0)$.  Since all points which are to the right of  $W^s(0,0)$ are in  $W^s(\alpha)$ by Theorem \ref{thwsloc}, obviously  $(x, y)$ is in $\partial W^s(\alpha)$. 
On the other hand, let $(x, y)$ be in $U$ and on  $\partial W^s(\alpha)$. If the point $(x, y)$ were not on $W^s(0,0)$, then it must be in the open set $W^s (\infty)$ contradicting the fact that it is on the boundary $\partial W^s(\alpha)$.
\end{proof}
To prove part (iii) of Theorem 1, we now describe the backward orbits  of $F_{\delta, \mu}$ for $1 < \mu < 3$ and small $\delta$. The next lemma shows that the backward orbits of points on $W^s(0,0))\setminus\{0\}$ escape which will be proved by using the standard trapping regions $V^+$, $V^-$ induced by a closed square $V$ of side length $r$ centered at the origin for an appropriate $r$ (see [BS, Lemma 2.1] and [FM]). In particular, $W^u(\infty)\subset V^+$, $W^s(\infty)\subset V^-$.
\begin{lemma}\label{lem19}
  $W^s(0,0))\setminus \{(0,0)\}\subset W^u(\infty)$.
\end{lemma}
\begin{proof}
Let  $W^s_+(0,0)$ respectively  $W^s_-(0,0)$ denote the two connected
components of $W^s(0,0)\setminus \{0\}$ whose existence is guaranteed by the Stable Manifold
Theorem. We use the immersed  topology on these components.
Note that this topology does not necessarily
 coincide with the relative topology induced by the topology of
$\bR^2$.  Define
$V^+_s=W^s(0,0)\cap V^+$. Then $V^+_s\subset
W^u(\infty)$. Moreover, since $W^s(0,0)$ is the boundary of the basin of
attraction of an attracting fixed point, and this basin is an unbounded
set contained in $V\cup V^+$, we conclude that
$V^+_s\not=\emptyset$. By making the radius defining $V$ larger if
necessary we may assume that $V^+_s\cap W^s_+(0,0)$ and $V^+_s\cap
W^s_-(0,0)$ are connected sets. 

Let $F= F_{\delta, \mu}$. Since $F^{-2}(V^+_s)\subset V^+_s$, we conclude that
$\bigcup_{n=1}^\infty F^{2n}(V^+_s)$ and $\bigcup_{n=1}^\infty F^{2n}(V^-_s)$ two  increasing unions  of connected curves. Moreover, these two unions are disjoint.
 Here we also use the fact that the
stable eigenvalue of the saddle point $0$ is negative and hence
$W^s_+(0,0)$ and  $W^s_-(0,0)$ are $F^{2n}$-invariant sets.
Consider $p\in W^s(0,0)$ with $p\not=0$. To prove the claim it suffices
to show that $p$ is contained in one of these two unions. Moreover, since $F^k(p)\in W^s_{\rm loc}(0,0)$
for some $k\in \bN$ it is enough to consider the case $p\in W^s_{\rm
loc}(0)$. Let now $n\in \bN$ be such that $F^{2n}(V^+_s)$ contains points
$p_1,p_2\in W^s_{\rm loc}(0,0)$ on both sides of $(0,0)$ which are closer to
$(0,0)$ than $p$. It now follows from the Stable Manifold Theorem that $p$
must be contained in one of the two curves in $F^{2n}(V^+_s)$. This
implies that  $p\in W^u(\infty)$.
\end{proof}

\begin{lemma}Let $1 < \mu < 3$. There exists  $\delta^* > 0$ such that for all $0 < \delta < \delta^*$
the map $F_{\delta, \mu}$ satisfies property (iii) of Theorem 1, that is
 $ K = \{\alpha,(0,0)\}\cup (W^s(\alpha)\cap W^u(0,0))$.
 \end{lemma}
 \begin{proof}
Obviously $W^s(\alpha)\cap W^u(0,0)$ is contained in  $ K $ as well as are the fixed points $\alpha$ and $(0,0)$. Let $q \in  K \setminus\{\alpha,(0,0)\}$.
By Propositions 13 and 16 the forward orbit of $q$ must converge to either $(0,0)$ or $\alpha$, i.e. $q $ is in $ W^s(0,0)$ but is not the origin or  $q $ is in $ W^s(\alpha)$ but is not $\alpha$. However, Lemma 18 implies that only the second case is possible, since the backward orbit of $q$ is bounded.

Now let $q$ be in the basin of attraction $W^s(\alpha)$ as well as in $K$ with $q \not=\alpha$. We will show that $q \in W^u(0,0)$.
Let $q_1$ be an arbitrary accumulation point of the backwards orbit $(F^{-n}(q))$, where $F = F_{\delta, \mu}$. Since $K$ is closed we conclude that $q_1 \in K$. Hence the forward orbit of $q_1$ must converge to either $(0,0)$, or $ \alpha $, again by Propositions 13 and 16.

 If $q_1\in W^s(0,0)$, then $q_1 = (0,0)$, because  otherwise $q_1\in W^u(\infty)$ by Lemma 18 which  contradicts the fact that $q_1 \in K$.
Since $q_1= (0,0)$ is an arbitrary accumulation point of  $F^{-n}(q)$, it follows that $q \in W^u(0,0)$ once we  show that $q_1\in W^s(\alpha)$ is not possible.

If $q_1 \in W^s(\alpha), q \not = \alpha$, then the backward orbit $(F^{-n}(q))_{n\geq 0}$ cannot have an accumulation point, because there is a sequence of mutually disjoint sets $A_n$ with  $F^{-1}(A_n) \subset A_{n+1}$ such that $W^s(\alpha) = \bigcup_{n\geq 0} A_n$.
\end{proof}

\begin{remarks} (i) It can be shown that the obtained $\delta^*$ can be derived as a continuous function of $\mu$. Therefore, if $1<\mu_0<3$ we can formulate Theorem 1 to be true for all parameters in the set $\{(\delta,\mu): 0<\delta<\delta^*, |\mu-\mu_0|<r\}$ for some $\delta^*>0$ and some $r>0$. \\
(ii)
Computer experiments suggest that Theorem 1 might be true for all $(\delta,\mu)\in H_{1,1}$ (see \eqref{H11} for the definition). However, our results are based on small perturbation techniques and do not apply to consider arbitrary parameters in $H_{1,1}$.
\end{remarks}  

\section{The general case}
In this section we consider  general maps of the form
\begin{equation}\label{eqgenfam1}
F(x,y)=(g(x)+h(y),h(x)).
\end{equation}
Here $g,h:\bR\to\bR$ are $C^2$-maps, where $g$ is a unimodal map whose graph has the qualitative shape of the logistic function $x\mapsto \mu x(1-x)$, $1<\mu<3$ and $h$ is a small perturbation of the linear map $L_\delta$. More precisely, we assume that $g:\bR\to \bR$ is a $C^2$-unimodal map with $g(0)=0, g(1)=0, g([0,1])\subset [0,1)$, $g'(0)>1$ and $g'(1)<-1$. We also assume that there exists $\gamma<0$ such that $g''(x)<\gamma$ for all $x\in \bR\setminus (0,1)$. Finally, we assume that there exists an attracting fixed point $x=x_g\in (0,1)$ of $g$ with $W^s(x)=(0,1)$.  For $\delta\in \bR$ we denote by $L_\delta$ the linear map $x\mapsto \delta x$.
We now discuss the proof of Theorem \ref{thmgeneral}. 
Since the arguments for the general case are obvious adaptations of the case for H\'enon maps, only a sketch will be outlined.

\noindent
{\it Proof of Theorem {\rm 3}. } 
(i)  If $\delta>0$ and $0<\epsilon<\delta/2$ it follows that every $C^2$-map $h:\bR\to\bR$ with $||h-L_\delta||_2<\epsilon$  is a $C^2$-diffeomorphism of $\bR$. Therefore,
it follows from \eqref{eqgenfam1} that $F$ is a bijective $C^2$-map of $\bR^2$. The statement that $F$ is a $C^2$-diffeomorphism of $\bR^2$ is now a consequence of the inverse mapping theorem.\\
(ii),(iii),(iv)  First, we note that if $\delta> 0$ is small enough then the origin is a saddle fixed point of $F$ and $F$ has an attracting fixed point $\alpha=\alpha_F$ that is close to $(x_g,0)$. Since $g((1,\infty))\subset (-\infty,0)$ and for all $x<0$ we have that $g^n(x)$ converges to $-\infty$ at a geometric rate,  it is straight forward to verify that analogous filtration properties to those in Lemmas 5 and 6 and Proposition 14 hold. Furthermore, since $W^s(x_g)=(0,1)$,  we can show that for small $\delta> 0$ the set $A_\delta$ (defined in \eqref{Adelta}) is contained in the basin of attraction of $\alpha$. The two key
results (Theorems \ref{thwsloc} and \ref{thws2}) which describe the iterates of $F$ near the origin (respectively near the intersection of $F^{-1}(W^s_{\rm loc}(0,0))$ with the $x$-axis) are  based on the geometric shape of the graph of $g$ near zero and can be proven accordingly. Finally, the corresponding results to Lemmas 17, 18 and 19 can be proven using the qualitative shape of $g,g'$ and $g''$ rather than the explicit formulas. These results show that Theorem 3 can be established by repeating the arguments of the proof of Theorem 1.

\end{document}